\documentclass{amsart}
\usepackage{txfonts}
\usepackage{}

\newtheorem{theorem}{Theorem}[section]
\newtheorem{lemma}[theorem]{Lemma}
\newtheorem{corollary}[theorem]{corollary}

\theoremstyle{definition}

\theoremstyle{remark}
\newtheorem{remark}[theorem]{Remark}

\numberwithin{equation}{section}

\begin{document}

\title{On compact manifolds with harmonic curvature and positive scalar curvature}

\author{Hai-Ping Fu}
\address{Department of Mathematics,  Nanchang University, Nanchang, P.
R. China 330031}
\email{mathfu@126.com}
\thanks{Supported by National Natural Science Foundations of China \#11261038 and \#11361041,  Jiangxi Province
Natural Science Foundation of China \#20132BAB201005.}






\begin{abstract}
Let $M^n(n\geq3)$ be an $n$-dimensional   compact Riemannian manifold  with harmonic curvature and  positive scalar curvature. Assume that $M^n$ satisfies some integral pinching conditions. We give some rigidity theorems on compact manifolds with harmonic curvature and positive scalar curvature. In particular, Theorem 1.4, Corollary 1.6 and Theorem 1.9 are sharp for our conditions have the additional properties of being sharp. By this we mean that we can precisely characterize the case of equality.

\end{abstract}

\maketitle

\section{Introduction and main results}
Recall that an $n$-dimensional Riemannian manifold $(M^n, g)$  is said to be  a manifold with harmonic curvature
if the divergence of its Riemannian curvature tensor $Rm$ vanishes, i.e., $\delta Rm=0$.
In view of the second Bianchi identity, we know that $M$ has harmonic
curvature if and only if the Ricci tensor of $M$ is a Codazzi tensor. When $n\geq 3$, by the Bianchi identity, the scalar curvature is constant. Thus, every Riemannian manifold with parallel Ricci tensor has harmonic
curvature. Moreover, the constant curvature spaces, Einstein manifolds and the locally conformally
flat manifolds with constant scalar curvature are
also important examples of manifolds with harmonic curvature, however, the
converse does not hold (see \cite{{B}}, for example). According to the decomposition of the Riemannian curvature tensor, the metric with harmonic curvature is a natural candidate for this study  since one of the important problems in Riemannian geometry is to understand classes
of metrics that are, in some sense, close to being Einstein or having constant curvature. The another reason for this study on the metric with harmonic curvature is the fact that a Riemannian manifold has harmonic curvature if and
only if the Riemannian connection is a solution of the Yang-Mills equations on
the tangent bundle \cite{Bo}. The complete manifolds  with
harmonic curvature have been studied in literature (e.g., \cite{{C},{Ca},{Ch},{D},{FX},{HV},{IS},{K},{PRS},{S},{TV},{XZ}}).
Some isolation theorems of Weyl curvature tensor of positive Einstein manifolds are given in \cite{{Ca2},{FX2},{FX3},{HV},{IS},{S}}, when its $L^{p}$-norm is small.
Some scholars \cite{{Ca},{HV},{PRS},{XZ}} classify  conformally flat  manifolds  satisfying an $L^{p}$-pinching condition on the curvature. Recently,
Xiao and the author \cite{FX} obtain some rigidities on complete manifolds with harmonic curvature satisfying an $L^{p}$-pinching condition
on trace-free Riemannian  curvature.
 The curvature pinching phenomenon plays an important role in global differential
geometry. We are interested in  $L^p$
pinching problems for compact Riemannian manifold  with harmonic curvature and  positive scalar curvature.

Now we  introduce the definition of the Yamabe constant. Given a compact  Riemannian $n$-manifold $M$, we consider the Yamabe functional
$$Q_g\colon C^{\infty}_{+}(M)\rightarrow\mathbb{R}\colon f\mapsto Q_g(f)=\frac{\frac{4(n-1)}{n-2}\int_M|\nabla f|^2\mathrm{d}v_g+\int_M Rf^2\mathrm{d}v_g}{(\int_M f^{\frac{2n}{n-2}}\mathrm{d}v_g)^{\frac{n-2}{n}}},$$
where $R$ denotes the  constant scalar curvature of $M$.
It follows that $Q_g$ is bounded below by H\"{o}lder inequality. We set
$$\mu([g])=\inf\{Q_g(f)|f\in C^{\infty}_{+}(M)\}.$$
This constant $\mu([g])$ is an invariant of the conformal class of $(M, g)$, called the Yamabe constant.
The important works of Schoen, Trudinger and Yamabe showed that the infimum in the above is always achieved (see \cite{{A},{LP}}).
The Yamabe constant of a given compact manifold
is determined by the sign of scalar curvature \cite{A}.

 Throughout this paper, we always assume
that $M$ is an $n$-dimensional complete  Riemannian manifold with $n\geq3$.
In this note, we obtain the following rigidity theorems.

\begin{theorem}
Let $M$  be an $n$-dimensional   compact Riemannian manifold  with harmonic curvature and  positive scalar curvature. Then
$$\int_M|\mathring{Ric}|^{\frac{n-2}{n}}[R-\sqrt{(n-1)n}|\mathring{Ric}|-\sqrt{\frac{(n-2)(n-1)}{2}}|W|]\leq 0$$
and equality occurs if and only if

i) $M$ is a Einstein manifold;

ii) $M$ is covered isometrically by $\mathbb{S}^1\times \mathbb{S}^{n-1}$ with the product metric;

iii) $M$ is covered isometrically by $(\mathbb{S}^1\times N^{n-1}, dt^2+F^2(t)g_N)$, where $(N^{n-1}, g_N)$is a
compact Einstein manifold with positive scalar curvature and $F$ is a non-constant, positive,
periodic function satisfying a precise ODE. This metric is called a rotationally symmetric Derdzi\'{n}ski metric in \cite{Ca}.
\end{theorem}

\begin{theorem}
Let $M$  be an $n$-dimensional   compact Riemannian manifold  with harmonic curvature and  positive scalar curvature.
Then
$$\int_M|\mathring{Rm}|[R-(n-1)C_1(n)|\mathring{Rm}|]\leq 0,$$
where $C_1(n)$ is defined in Lemma 2.1, and equality occurs if and only if $M$ is isometric to a quotient of the round $\mathbb{S}^n$.
\end{theorem}
\begin{corollary}
Let $M$ be a compact  Riemannian $n$-manifold
with harmonic curvature and positive scalar curvature. If
$$|\mathring{Rm}|\leq\frac{R}{(n-1)C_1(n)},$$
then $M$ is isometric to a quotient of the round $\mathbb{S}^n$.
\end{corollary}

\begin{theorem}
Let $(M^n, g)(n\geq4)$  be an $n$-dimensional   compact Riemannian manifold  with harmonic curvature and  positive scalar curvature. If
\begin{equation}
\begin{split}\left(\int_M|W+\frac{\sqrt{n}}{\sqrt{2}(n-2)}\mathring{Ric}\circledwedge g|^{\frac n2}\right)^{\frac 2n}
<\frac{\sqrt{2}}{\sqrt{(n-2)(n-1)}}\mu([g]),
\end{split}
\end{equation}
then $M$ is an Einstein manifold.
In particular, if the pinching constant in (1.1) is weakened to $\frac{2}{nC_2(n)}\mu([g])$, where $C_2(n)$ is defined in Lemma 2.4, then $M$ is isometric to a quotient of the round $\mathbb{S}^n$.
\end{theorem}
\begin{remark}Theorem 1.4 improves Theorems 1 and 2 in \cite{HV}. When $n=3$, these manifolds with harmonic curvature are locally conformally
flat. Theorem 1.4 in dimension 3 is discussed by Xiao and the author in \cite{FX}. The inequality (1.1) of this theorem is optimal.
The critical case is given by the following example. If $(\mathbb{S}^1(t)\times \mathbb{S}^{n-1}, g_t)$ is the product of the circle
of radius $t$ with $\mathbb{S}^{n-1}$, and if $g_t$ is the standard product metric normalized such that $Vol(g_t)=1$, we have
$W=0$, $g_t$ is a Yamabe metric for small $t$ (see \cite{S1}), and $\left(\int_M|\mathring{Ric}|^{\frac n2}\right)^{\frac 2n}=\frac{\mu([g_t])}{\sqrt{n(n-1)}}$, which is the critical case  of the inequality (1.1) in Theorem 1.4. We know that $(\mathbb{S}^1(t)\times \mathbb{S}^{n-1}, g_t)$ is not Einstein.
\end{remark}

\begin{corollary}
Let $(M^4, g)$  be a $4$-dimensional   compact Riemannian manifold  with harmonic curvature and  positive scalar curvature. If
\begin{equation}\int_M|W|^{2}+8\int_M|\mathring{Ric}|^2
\leq\frac{1}{3}\int_M R^2,\end{equation}
then i)$M$ is isometric to a quotient of the round $\mathbb{S}^4$;

ii) $M$ is a compact positive Einstein manifolds which is either K\"{a}hler, or the quotient of a K\"{a}hler manifold by a free, isometric,
anti-holomorphic involution;

iii) $M$ is a $\mathbb{CP}^2$ with the Fubini-Study metric;

iv) $M$ is Einstein and self-dual.

In particular, if \begin{equation}\int_M|W|^{2}+6\int_M|\mathring{Ric}|^2
\leq\frac{1}{6}\int_M R^2,\end{equation} then  i), ii) and iii) hold.
\end{corollary}
\begin{remark}The pinching conditions (1.2) and (1.3) in Corollary 1.6 is equivalent to the following
\begin{equation*}\int_M|W|^{2}+\frac {1}{15}\int_M R^2
\leq\frac{128}{5}{\pi}^2\chi(M)\end{equation*}
and
\begin{equation*}\int_M|\mathring{Rm}|^{2}
\leq16{\pi}^2\chi(M),\end{equation*}
where $\chi(M)$ is the Euler-Poincar\'{e} characteristic of $M$.
\end{remark}

As we mentioned above, Theorem 1.4 is sharped. By this we mean that we can precisely characterize the case of equality:
\begin{theorem}
Let $(M^n, g)(n\geq4)$  be an $n$-dimensional   compact Riemannian manifold  with harmonic curvature and  positive scalar curvature. If
\begin{equation}
\begin{split}\left(\int_M|W+\frac{\sqrt{n}}{\sqrt{2}(n-2)}\mathring{Ric}\circledwedge g|^{\frac n2}\right)^{\frac 2n}
=\frac{\sqrt{2}}{\sqrt{(n-2)(n-1)}}\mu([g]),
\end{split}
\end{equation}
then
 i) $M$ is Einstein manifold;

ii) $M$ is covered isometrically by $\mathbb{S}^1\times \mathbb{S}^{n-1}$ with the product metric;

iii) $M$ is covered isometrically by $\mathbb{S}^1\times \mathbb{S}^{n-1}$ with a rotationally symmetric Derdzi\'{n}ski metric.
\end{theorem}

\begin{theorem}
Let $(M^n, g)(n\geq4)$  be an $n$-dimensional   compact Riemannian manifold  with harmonic curvature and  positive scalar curvature. If
\begin{equation}
\begin{split}C_2(n)\left(\int_M|W|^{\frac n2}\right)^{\frac 2n}+2\sqrt{\frac{n-1}{n}}\left(\int_M|\mathring{Ric}|^{\frac n2}\right)^{\frac 2n}
<D(n)\mu([g]),
\end{split}
\end{equation}
where \begin{equation*}D(n)=
\begin{cases}\frac{n-2}{4(n-1)}, n\leq9;\\ \frac{2}{n}, n\geq 10,
\end{cases}\end{equation*} then $M$ is isometric to a quotient of the round $\mathbb{S}^n$. Moreover, in dimension, $4\leq n\leq9$,
the same result holds only assuming the weak inequality.
\end{theorem}
\begin{remark}When $n\geq 10$, the pinching condition  of this theorem is optimal.
The critical case is given by this  example in Remark 1.5.
\end{remark}

As we mentioned above, Theorem 1.9 is sharped. By this we mean that we can precisely characterize the case of equality:
\begin{theorem}
Let $(M^n, g)(n\geq10)$  be an $n$-dimensional   compact Riemannian manifold  with harmonic curvature and  positive scalar curvature. If
\begin{equation}
\begin{split}C_2(n)\left(\int_M|W|^{\frac n2}\right)^{\frac 2n}+2\sqrt{\frac{n-1}{n}}\left(\int_M|\mathring{Ric}|^{\frac n2}\right)^{\frac 2n}
=\frac{2}{n}\mu([g]),
\end{split}
\end{equation}
then
i) $M$  is Einstein manifold;

ii) $M$ is covered isometrically by $\mathbb{S}^1\times \mathbb{S}^{n-1}$ with the product metric;

iii) $M$ is covered isometrically by $\mathbb{S}^1\times \mathbb{S}^{n-1}$ with a rotationally symmetric Derdzi\'{n}ski metric.
\end{theorem}

Based on Lemma 2.4, using the same argument as in the proof of Theorem 1.1, we can prove Theorem 1.12.

\begin{theorem}
Let $M$  be an $n(\geq 4)$-dimensional   compact Riemannian manifold  with harmonic curvature and  positive scalar curvature. Then
$$\int_M|W|[R-\sqrt{(n-1)n}|\mathring{Ric}|-\frac{C_2(n)n}{2}|W|]\leq 0$$
and equality occurs if and only if
$M$ is  locally conformal flat.
\end{theorem}
\begin{remark} G. Catino \cite{Ca}  classifies compact conformally
at $n$-dimensional manifolds with constant
positive scalar curvature and satisfying an optimal integral pinching condition.
\end{remark}

\begin{remark} When $n=4$, we \cite{F} extend to four-manifolds with harmonic Weyl tensor Theorems 1.4 and 1.8, and corollary 1.6.
\end{remark}

We  follow their methods \cite{{Ca},{Ca2},{FX2},{HL}} to prove these theorems.


\section{Proofs of Lemmas}
In what follows, we adopt, without further comment, the moving frame notation with respect to a chosen local orthonormal frame.

Let $M$ be a  Riemannian manifold with harmonic curvature. The
decomposition of the Riemannian curvature tensor  into irreducible components yield
\begin{eqnarray*}
R_{ijkl}&=&W_{ijkl}+\frac{1}{n-2}(R_{ik}\delta_{jl}-R_{il}\delta_{jk}+R_{jl}\delta_{ik}-R_{jk}\delta_{il})\nonumber\\
&&-\frac{R}{(n-1)(n-2)}(\delta_{ik}\delta_{jl}-\delta_{il}\delta_{jk})\nonumber\\
&=&W_{ijkl}+\frac{1}{n-2}(\mathring{R}_{ik}\delta_{jl}-\mathring{R}_{il}\delta_{jk}+\mathring{R}_{jl}\delta_{ik}-\mathring{R}_{jk}\delta_{il})\nonumber\\
&&+\frac{R}{n(n-1)}(\delta_{ik}\delta_{jl}-\delta_{il}\delta_{jk}),
\end{eqnarray*}
where $R_{ijkl}$, $W_{ijkl}$, $R_{ij}$  and $\mathring{R}_{ij}$  denote the components of $Rm$, the Weyl curvature tensor $W$,  the Ricci tensor $Ric$ and the trace-free Ricci tensor $\mathring{Ric}=Ric-\frac{R}{n}g$, respectively,  and  $R$  is the scalar curvature.

The trace-free Riemannian curvature tensor $\mathring{Rm}$ is
\begin{equation}\mathring{R}_{ijkl}={R}_{ijkl}-\frac{R}{n(n-1)}(\delta_{ik}\delta_{jl}-\delta_{il}\delta_{jk}).\end{equation}
Then the following equalities are easily obtained
from the properties of curvature tensor:
\begin{equation}g^{ik}\mathring{R}_{ijkl}=\mathring{R}_{jl},\end{equation}
\begin{equation}\mathring{R}_{ijkl}+\mathring{R}_{iljk}+\mathring{R}_{iklj}=0,\end{equation}
\begin{equation}\mathring{R}_{ijkl}=\mathring{R}_{klij}=-\mathring{R}_{jikl}=-\mathring{R}_{ijlk},\end{equation}
\begin{equation}|\mathring{Rm}|^2=|W|^2+\frac{4}{n-2}|\mathring{Ric}|^2.\end{equation}
Moreover, by the assumption of harmonic curvature,
we compute
\begin{equation}
\mathring{R}_{ijkl,m}+\mathring{R}_{ijmk,l}+\mathring{R}_{ijlm,k}=0,
\end{equation}
\begin{equation}
\mathring{R}_{ijkl,l}=0,
\end{equation}
\begin{equation}
\mathring{W}_{ijkl,m}+\mathring{W}_{ijmk,l}+\mathring{W}_{ijlm,k}=0
\end{equation}
and
\begin{equation}
\mathring{W}_{ijkl,l}=0,
\end{equation}

Now, we compute the Laplacian of $|\mathring{Rm}|^2$, $|\mathring{Ric}|^2$ and $|\mathring{W}|^2$, respectively.
\begin{lemma}
Let $M$  be a complete  Riemannian $n$-manifold
with harmonic curvature and positive scalar curvature.   Then
\begin{equation}
\triangle|\mathring{Rm}|^2\geq2|\nabla \mathring{Rm}|^2-2C_1(n)|\mathring{Rm}|^3+2\frac{(n-2)R}{n(n-1)}|W|^2+2\frac{R}{n-1}|\mathring{Rm}|^2,
\end{equation}
where $C_1(n)=2[\frac{2(n^2+n-4)}{\sqrt{(n-1)n(n+1)(n+2)}}+ \frac{n^2-n-4}{2\sqrt{(n-2)(n-1)n(n+1)}}+\sqrt{\frac{(n-2)(n-1)}{4n}}]$.
\end{lemma}
\begin{remark} Although Lemma 2.1 has been proved in \cite{FX}, we give its proof.
When $M$ is a  complete locally
conformally flat Riemannian $n$-manifold, it follows from (2.12) that
\begin{equation*} \triangle|\mathring{Ric}|^2\geq2|\nabla \mathring{Ric}|^2-2\frac{n}{\sqrt{n(n-1)}}|\mathring{Ric}|^3+2\frac{R}{n-1}|\mathring{Ric}|^2.\end{equation*}
By the Kato inequality $|\nabla \mathring{Ric}|^2\geq \frac{n+2}{n}|\nabla |\mathring{Ric}||^2$, we obtain ( see \cite{{PRS},{XZ}})
\begin{equation} |\mathring{Ric}|\triangle|\mathring{Ric}|\geq\frac{2}{n}|\nabla |\mathring{Ric}||^2-\frac{n}{\sqrt{n(n-1)}}|\mathring{Ric}|^3+\frac{R}{n-1}|\mathring{Ric}|^2.\end{equation}
\end{remark}
\begin{proof}
By the Ricci identities, we obtain from (2.1)-(2.7)
\begin{equation}
\begin{split}
\triangle|\mathring{Rm}|^2&=2|\nabla \mathring{Rm}|^2+2\langle \mathring{Rm}, \triangle \mathring{Rm}\rangle=2|\nabla \mathring{Rm}|^2+2\mathring{R}_{ijkl}\mathring{R}_{ijkl,mm}\\
&=2|\nabla \mathring{Rm}|^2+2\mathring{R}_{ijkl}(\mathring{R}_{ijkm,lm}+\mathring{R}_{ijml,km})\\
&=2|\nabla \mathring{Rm}|^2+4\mathring{R}_{ijkl}\mathring{R}_{ijkm,lm}\\
&=2|\nabla \mathring{Rm}|^2+4\mathring{R}_{ijkl}(\mathring{R}_{ijkm,ml}
+\mathring{R}_{hjkm}R_{hilm}\\
&+\mathring{R}_{ihkm}R_{hjlm}+\mathring{R}_{ijhm}R_{hklm}
+\mathring{R}_{ijkh}R_{hmlm})\\
&=2|\nabla \mathring{Rm}|^2+4\mathring{R}_{ijkl}(\mathring{R}_{hjkm}R_{hilm}
+\mathring{R}_{ihkm}R_{hjlm}\\
&+\mathring{R}_{ijhm}R_{hklm}
+\mathring{R}_{ijkh}R_{hmlm})\\
&=2|\nabla \mathring{Rm}|^2+4\mathring{R}_{ijkl}(\mathring{R}_{hjkm}\mathring{R}_{hilm}
+\mathring{R}_{ihkm}\mathring{R}_{hjlm}
+\mathring{R}_{ijhm}\mathring{R}_{hklm}\\
&+\mathring{R}_{ijkh}\mathring{R}_{hmlm})
+\frac{4R}{n(n-1)}\mathring{R}_{ijkl}(\mathring{R}_{ljki}+\mathring{R}_{ilkj}+\mathring{R}_{ijlk}\\
&+\mathring{R}_{jk}\delta_{il}-\mathring{R}_{ik}\delta_{jl})+\frac{4R}{n}|\mathring{Rm}|^2\\
&=2|\nabla \mathring{Rm}|^2+4\mathring{R}_{ijkl}(\mathring{R}_{hjkm}\mathring{R}_{hilm}
+\mathring{R}_{ihkm}\mathring{R}_{hjlm}
+\mathring{R}_{ijhm}\mathring{R}_{hklm}\\
&+\mathring{R}_{ijkh}\mathring{R}_{hl})
-\frac{8R}{n(n-1)}|\mathring{Ric}|^2+\frac{4R}{n}|\mathring{Rm}|^2\\
&=2|\nabla \mathring{Rm}|^2-4\mathring{R}_{ijlk}(2\mathring{R}_{jhkm}\mathring{R}_{himl}
+\frac 12\mathring{R}_{hmij}\mathring{R}_{lkhm}\\
&+\mathring{R}_{ijkh}\mathring{R}_{hl})
+\frac{2(n-2)R}{n(n-1)}|W|^2+\frac{2R}{n-1}|\mathring{Rm}|^2.
\end{split}
\end{equation}

We consider $\mathring{Rm}$ as a self adjoint operator on $\wedge^2 V$ and $S^2 V$, respectively. By the algebraic inequality for $m$-trace-free symmetric two-tensors $T$, i.e., $tr(T^3)\leq\frac{m-2}{\sqrt{m(m-1)}}|T|^3$ and equality holds if and only if $T$ can be diagonalized  with
$(n-1)$-eigenvalues equal to $\lambda$ and one eigenvalue equals to $-(n-1)\lambda$, and the eigenvalues $\lambda_i$ of $T$ satisfy $|\lambda_i|\leq\sqrt{\frac{m-1}{m}}|T|$ in \cite{Hu}, we obtain
\begin{equation}
\begin{split}|\mathring{R}_{ijlk}(2\mathring{R}_{jhkm}\mathring{R}_{himl}
+\frac 12\mathring{R}_{hmij}\mathring{R}_{lkhm})|
\leq2|\mathring{R}_{ijlk}\mathring{R}_{jhkm}\mathring{R}_{himl}|
+\frac 12|\mathring{R}_{ijlk}\mathring{R}_{hmij}\mathring{R}_{lkhm}|\\
\leq[\frac{2(n^2+n-4)}{\sqrt{(n-1)n(n+1)(n+2)}}+ \frac{n^2-n-4}{2\sqrt{(n-2)(n-1)n(n+1)}}]|\mathring{Rm}|^3,
\end{split}
\end{equation}
and
\begin{eqnarray}|\mathring{R}_{ijkl}\mathring{R}_{ijkh}\mathring{R}_{hl}|\leq\sqrt{\frac{n-1}{n}}|\mathring{Ric}||\mathring{Rm}|^2.
\end{eqnarray}
From (2.5), we have
\begin{equation}|\mathring{Ric}|^2\leq \frac{n-2}{4}|\mathring{Rm}|^2.
\end{equation}
Combining  with (2.12)-(2.15), we obtain  that
\begin{equation*}
\begin{split}
\triangle|\mathring{Rm}|^2\geq 2|\nabla \mathring{Rm}|^2+\frac{2(n-2)R}{n(n-1)}|W|^2+\frac{2R}{n-1}|\mathring{Rm}|^2-4[\sqrt{\frac{(n-2)(n-1)}{4n}}
\\+\frac{2(n^2+n-4)}{\sqrt{(n-1)n(n+1)(n+2)}}+ \frac{n^2-n-4}{2\sqrt{(n-2)(n-1)n(n+1)}}]|\mathring{Rm}|^3.
\end{split}
\end{equation*}
This completes the proof of this Lemma.
\end{proof}

\begin{lemma}
Let $M$  be a complete  Riemannian $n$-manifold
with harmonic curvature.   Then
\begin{equation}
\triangle|\mathring{Ric}|^2\geq2|\nabla \mathring{Ric}|^2-2\sqrt{\frac{n-2}{2(n-1)}}|W||\mathring{Ric}|^2-2\sqrt{\frac{n}{n-1}}|\mathring{Ric}|^3+2\frac {R} {n-1}|\mathring{Ric}|^2.
\end{equation}
\end{lemma}
\begin{proof}
We compute
\begin{equation}
\triangle|\mathring{Ric}|^2=2|\nabla \mathring{Ric}|^2+2\langle \mathring{Ric}, \triangle \mathring{Ric}\rangle=2|\nabla \mathring{Ric}|^2+2\mathring{R}_{ij}\mathring{R}_{ij,kk}.
\end{equation}
Since the traceless Ricci tensor is  Codazzi, by the Ricci identities, we obtain
\begin{equation}
\begin{split} \mathring{R}_{ij,kk}&=\mathring{R}_{ik,jk}
=\mathring{R}_{ki,kj}+\mathring{R}_{li}R_{lkjk}+\mathring{R}_{kl}R_{lijk}\\
&=\mathring{R}_{kk,ij}+\mathring{R}_{li}R_{lkjk}+\mathring{R}_{kl}R_{lijk}\\
&=\mathring{R}_{li}R_{lkjk}+\mathring{R}_{kl}R_{lijk},
\end{split}
\end{equation}
which gives
\begin{equation}
\triangle|\mathring{Ric}|^2=2|\nabla \mathring{Ric}|^2+2\mathring{R}_{ij}\mathring{R}_{ij,kk}=2|\nabla \mathring{Ric}|^2+2\mathring{R}_{ij}\mathring{R}_{li}R_{lkjk}+2\mathring{R}_{ij}\mathring{R}_{kl}R_{lijk}.\end{equation}
We compute
\begin{equation}
\triangle|\mathring{Ric}|^2=2|\nabla \mathring{Ric}|^2+2W_{kijl}\mathring{R}_{ij}\mathring{R}_{kl}+2\frac{n}{n-2}\mathring{R}_{ij}\mathring{R}_{jl}\mathring{R}_{li}
+2\frac{R}{n-1}|\mathring{Ric}|^2.
\end{equation}
By the inequality $|W_{kijl}\mathring{R}_{ij}\mathring{R}_{kl}|\leq\sqrt{\frac{n-2}{2(n-1)}}|W||\mathring{Ric}|^2$ given in \cite{Hu}, from (2.20) we get
\begin{equation*}
\triangle|\mathring{Ric}|^2\geq2|\nabla \mathring{Ric}|^2-2\sqrt{\frac{n-2}{2(n-1)}}|W||\mathring{Ric}|^2-2\sqrt{\frac{n}{n-1}}|\mathring{Ric}|^3+2\frac {R}{n-1}|\mathring{Ric}|^2.
\end{equation*}
This completes the proof of this Lemma.
\end{proof}

\begin{lemma}
Let $M$  be a complete  Riemannian $n$-manifold
with harmonic curvature.      Then
\begin{eqnarray}
\triangle|W|^2\geq2|\nabla W|^2-2C_2(n)|W|^3-4\sqrt{\frac{n-1}{n}}|W|^2|\mathring{Ric}|+\frac{4R}{n}|W|^2,
\end{eqnarray}
where\begin{equation*}C_2(n)=
\begin{cases}\frac{\sqrt{6}}{2}, \quad n=4 \\\frac{8\sqrt{10}}{15}, \quad n=5\\\frac{4(n^2+n-4)}{\sqrt{(n-1)n(n+1)(n+2)}}+ \frac{n^2-n-4}{\sqrt{(n-2)(n-1)n(n+1)}}, \quad  n\geq6.
\end{cases}
\end{equation*}
\end{lemma}
\begin{proof}
By the Ricci identities, we obtain form (2.8) and (2.9)
\begin{equation}
\begin{split}
\triangle|W|^2&=2|\nabla W|^2+2\langle W, \triangle W\rangle=2|\nabla W|^2+2W_{ijkl}W_{ijkl,mm}\\
&=2|\nabla W|^2+2W_{ijkl}(W_{ijkm,lm}+W_{ijml,km})\\
&=2|\nabla W|^2+4W_{ijkl}W_{ijkm,lm}\\
&=2|\nabla W|^2+4W_{ijkl}(W_{ijkm,ml}
+W_{hjkm}R_{hilm}+W_{ihkm}R_{hjlm}+W_{ijhm}R_{hklm}
+W_{ijkh}R_{hmlm})\\
&=2|\nabla W|^2+4W_{ijkl}(W_{hjkm}R_{hilm}
+W_{ihkm}R_{hjlm}+W_{ijhm}R_{hklm}
+W_{ijkh}R_{hmlm})\\
&=2|\nabla W|^2+4W_{ijkl}(W_{hjkm}W_{hilm}
+W_{ihkm}W_{hjlm}
+W_{ijhm}W_{hklm}
+W_{ijkh}W_{hmlm})\\
&
+\frac{4}{n-2}W_{ijkl}[W_{hjkm}(\mathring{R}_{hl}\delta_{im}-\mathring{R}_{hm}\delta_{il}+\mathring{R}_{im}\delta_{hl}-\mathring{R}_{il}\delta_{hm})\\
&+W_{ihkm}(\mathring{R}_{hl}\delta_{jm}-\mathring{R}_{hm}\delta_{jl}+\mathring{R}_{jm}\delta_{hl}-\mathring{R}_{jl}\delta_{hm})+W_{ijhm}(\mathring{R}_{hl}\delta_{km}-\mathring{R}_{hm}\delta_{kl}+\mathring{R}_{km}\delta_{hl}-\mathring{R}_{kl}\delta_{hm})\\
&+W_{ijkh}(\mathring{R}_{hl}\delta_{mm}-\mathring{R}_{hm}\delta_{ml}+\mathring{R}_{mm}\delta_{hl}-\mathring{R}_{ml}\delta_{hm})]
+\frac{4R}{n(n-1)}W_{ijkl}(W_{ljki}+W_{ilkj}+W_{ijlk})
+\frac{4R}{n}|W|^2\\
&=2|\nabla W|^2+4W_{ijkl}(2W_{hjkm}W_{hilm}
-\frac 12W_{ijhm}W_{klhm})+\frac{4R}{n}|W|^2+4W_{ijkl}W_{ijkh}\mathring{R}_{hl}\\
&\geq2|\nabla W|^2-4(2W_{ijlk}W_{jhkm}W_{himl}
+\frac 12W_{ijkl}W_{hmij}W_{klhm})+\frac{4R}{n}|W|^2-4\sqrt{\frac{n-1}{n}}|W|^2|\mathring{Ric}|,
\end{split}
\end{equation}
and equality holds if and only if $W=0$ or $\mathring{Ric}$ can be diagonalized  with
$(n-1)$-eigenvalues equal to $\lambda$ and one eigenvalue equals to $-(n-1)\lambda$
and $(W_{ijk1}, \cdots, W_{ijkn})$ correspondingly takes the value $(-(n-1)\lambda, 0, \cdots, 0)$.

Case 1. When $n=4$, it was proved in \cite{Hu} that $$|2W_{ijlk}W_{jhkm}W_{himl}
+\frac 12W_{ijkl}W_{hmij}W_{klhm}|\leq\frac {\sqrt{6}}{4}|W|^3.$$

Case 2. When $n=5$,  Jack and  Parker \cite{JP} have proved that $W_{ijkl}W_{hmij}W_{klhm}=4W_{ijlk}W_{jhkm}W_{himl}$. We consider $W$ as a self adjoint operator on $\wedge^2 V$,  and obtain
\begin{eqnarray*}|2W_{ijlk}W_{jhkm}W_{himl}
+\frac 12W_{ijkl}W_{hmij}W_{klhm}|=|W_{ijkl}W_{hmij}W_{klhm}|\leq\frac {4\sqrt{10}}{15}|W|^3.\end{eqnarray*}

Case 3. When $n\geq 6$, considering $W$ as a self adjoint operator on $S^2 V$, we have
\begin{eqnarray*}|2W_{ijlk}W_{jhkm}W_{himl}
+\frac 12W_{ijkl}W_{hmij}W_{klhm}|\leq2|W_{ijlk}W_{jhkm}W_{himl}|+\frac 12|W_{ijkl}W_{hmij}W_{klhm}|\\
\leq[\frac{2(n^2+n-4)}{\sqrt{(n-1)n(n+1)(n+2)}}+ \frac{n^2-n-4}{2\sqrt{(n-2)(n-1)n(n+1)}}]|W|^3.\end{eqnarray*}

From (2.22) and Cases 1,2 and 3, we complete the proof of this Lemma.
\end{proof}
\begin{lemma}
On every $n$-dimensional Riemannian manifold the following estimate holds
\begin{eqnarray*}
\left|-W_{ijkl}\mathring{R}_{ik}\mathring{R}_{jl}+\frac{n}{n-2}\mathring{R}_{ij}\mathring{R}_{jk}\mathring{R}_{ki}\right|\leq
\sqrt{\frac{n-2}{2(n-1)}}\left(|W|^2+\frac{2n}{n-2}|\mathring{Ric}|^2\right)^{\frac12}|\mathring{Ric}|^2.
\end{eqnarray*}
\end{lemma}
\begin{remark}
We follow these proofs of Proposition 2.1 in \cite{Ca2} and Lemma 4.7 in \cite{B2} to  prove this lemma.
\end{remark}
\section{Proofs of Theorems}

\begin{proof}[{\bf Proof of Theorem
1.1}]
By the Kato inequality $|\nabla \mathring{Ric}|^2\geq \frac{n+2}{n}|\nabla |\mathring{Ric}||^2$ and Lemma 2.3, we obtain
\begin{equation}
|\mathring{Ric}|\triangle|\mathring{Ric}|\geq\frac{2}{n}|\nabla |\mathring{Ric}||^2-\sqrt{\frac{n-2}{2(n-1)}}|W||\mathring{Ric}|^2-\sqrt{\frac{n}{n-1}}|\mathring{Ric}|^3+\frac {R}{n-1}|\mathring{Ric}|^2
\end{equation}
in the sense of distributions. Using (3.1), we compute
\begin{equation}
\begin{split}
|\mathring{Ric}|^{\frac{n-2}{n}}\triangle |\mathring{Ric}|^{\frac{n-2}{n}}=|\mathring{Ric}|^{\frac{n-2}{n}}\left(\frac{-2(n-2)}{n^2}|\mathring{Ric}|^{\frac{-n-2}{n}}|\nabla |\mathring{Ric}||^2+\frac{n-2}{n}|\mathring{Ric}|^{\frac{-2}{n}}\triangle |\mathring{Ric}|\right)\\
=\frac{-2}{n-2}|\nabla |\mathring{Ric}|^{\frac{n-2}{n}}|^2+\frac{n-2}{n}
|\mathring{Ric}|^{\frac{-4}{n}}|\mathring{Ric}|\triangle |\mathring{Ric}|\\
\geq\frac{n-2}{n}\left(-\sqrt{\frac{n}{n-1}}|\mathring{Ric}|^{2\frac{n-2}{n}+1}+\frac {R}{n-1} |\mathring{Ric}|^{2\frac{n-2}{n}}-\sqrt{\frac{n-2}{2(n-1)}}|W||\mathring{Ric}|^{2\frac{n-2}{n}}\right),
\end{split}
\end{equation}
in the sense of distributions. From (3.2), we get
\begin{equation}\triangle |\mathring{Ric}|^{\frac{n-2}{n}}\geq\frac{n-2}{n}\left(-\sqrt{\frac{n}{n-1}}|\mathring{Ric}|^{\frac{n-2}{n}+1}+\frac {R}{n-1} |\mathring{Ric}|^{\frac{n-2}{n}}-\sqrt{\frac{n-2}{2(n-1)}}|W||\mathring{Ric}|^{\frac{n-2}{n}}\right).\end{equation}
Integrating (3.3) over $M$, we obtain
\begin{equation}\int_M\left(R-\sqrt{n(n-1)}|\mathring{Ric}| -\sqrt{\frac{(n-2)(n-1)}{2}}|W|\right)|\mathring{Ric}|^{\frac{n-2}{n}}\leq 0.\end{equation}

If the equality holds in (3.4), all inequalities leading to (3.1)
become equalities. Hence at every point, either $\mathring{Ric}$ is null, i.e., $M$ is Eninstein, or it has an eigenvalue of
multiplicity $(n-1)$ and another of multiplicity $1$. Since $M$ has harmonic curvature, and by the regularity result of DeTurck and Goldschmidt \cite{DG},
$M$ must be real analytic
in suitable (harmonic) local coordinates.

Suppose that the Ricci tensor has an eigenvalue of multiplicity $(n-1)$ and another of multiplicity $1$.
If the Ricci tensor is parallel, by the de Rham decomposition Theorem \cite{Dr},
$M$ is covered isometrically by the product of Einstein manifolds. We have $R=\sqrt{n(n-1)}|\mathring{Ric}|$. From (3.4), we get $W=0$, i.e. is conformally flat. Since $M$ has positive scalar curvature, then the only possibility is that $M$ is covered
isometrically by $\mathbb{S}^1\times \mathbb{S}^{n-1}$ with the product metric.

On the other hand, if the Ricci tensor is not parallel, by the  classification
result of Derdzi\'{n}ski (see Theorem 10 of \cite{D}, see also Theorem 3.2 of \cite{Ca}), this concludes the proof of Theorem 1.1.
\end{proof}

\begin{proof}[{\bf Proof of Theorem
1.2}]
By the Kato inequality $|\nabla \mathring{Rm}|^2\geq |\nabla |\mathring{Rm}||^2$ and Lemma 2.2, we get
\begin{equation}
\frac12\triangle|\mathring{Rm}|^2\geq |\nabla |\mathring{Rm}||^2-C_1(n)|\mathring{Rm}|^3+\frac{(n-2)R}{n(n-1)}|W|^2+\frac{R}{n-1}|\mathring{Rm}|^2
\end{equation}
in the sense of distributions.

For small $\epsilon>0$, define $\Omega_\epsilon=\{x\in M||\mathring{Rm}|\geq\epsilon\},$ and
\begin{equation*}f_\epsilon(x)=
\begin{cases}|\mathring{Rm}|(x) \quad x\in \Omega_\epsilon\\\epsilon \quad\quad \quad \ \   x\in M\setminus
\Omega_\epsilon.
\end{cases}
\end{equation*}
Multiplying both sides of (3.5) by $f^{-1}_\epsilon$ and then integrating over $M$, we obtain
\begin{equation}
\begin{split} 0&\geq-\frac12\int_M\triangle|\mathring{Rm}|^2f^{-1}_\epsilon+\int_M|\nabla |\mathring{Rm}||^2f^{-1}_\epsilon\\
&+\int_M [-C_1(n)|\mathring{Rm}|^3+\frac{(n-2)R}{n(n-1)}|W|^2+\frac{R}{n-1}|\mathring{Rm}|^2]f^{-1}_\epsilon\\
&=\frac12\int_M\langle\nabla|\mathring{Rm}|^2, \nabla f^{-1}_\epsilon\rangle+\int_M|\nabla |\mathring{Rm}||^2f^{-1}_\epsilon\\
&+\int_M [-C_1(n)|\mathring{Rm}|^3+\frac{(n-2)R}{n(n-1)}|W|^2+\frac{R}{n-1}|\mathring{Rm}|^2]f^{-1}_\epsilon\\
&=-\int_M\langle\nabla|\mathring{Rm}|, \nabla f_\epsilon\rangle |\mathring{Rm}|f^{-2}_\epsilon+\int_M|\nabla |\mathring{Rm}||^2f^{-1}_\epsilon\\
&+\int_M [-C_1(n)|\mathring{Rm}|^3+\frac{(n-2)R}{n(n-1)}|W|^2+\frac{R}{n-1}|\mathring{Rm}|^2]f^{-1}_\epsilon\\
&=-\int_M|\nabla f_\epsilon|^2f^{-1}_\epsilon+\int_M|\nabla |\mathring{Rm}||^2f^{-1}_\epsilon\\
&+\int_M [-C_1(n)|\mathring{Rm}|^3+\frac{(n-2)R}{n(n-1)}|W|^2+\frac{R}{n-1}|\mathring{Rm}|^2]f^{-1}_\epsilon\\
&=\int_{M\setminus
\Omega_\epsilon}|\nabla |\mathring{Rm}||^2f^{-1}_\epsilon+\int_M [-C_1(n)|\mathring{Rm}|^3+\frac{(n-2)R}{n(n-1)}|W|^2+\frac{R}{n-1}|\mathring{Rm}|^2]f^{-1}_\epsilon\\
&\geq\int_M [-C_1(n)|\mathring{Rm}|^3+\frac{(n-2)R}{n(n-1)}|W|^2+\frac{R}{n-1}|\mathring{Rm}|^2]f^{-1}_\epsilon\\
&\geq\int_M [-C_1(n)|\mathring{Rm}|^3+\frac{R}{n-1}|\mathring{Rm}|^2]f^{-1}_\epsilon.
\end{split}
\end{equation}
It follows from the proof of (3.6) that
$$\int_M [R-C_1(n)(n-1)|\mathring{Rm}|]|\mathring{Rm}|^2f^{-1}_\epsilon\leq0.$$
Now, taking the limit as $\epsilon\rightarrow 0$, we get form the above inequality
\begin{equation}\int_M [R-C_1(n)(n-1)|\mathring{Rm}|]|\mathring{Rm}|\leq0.\end{equation}

If the equality holds in (3.7), all inequalities leading to (3.5)
become equalities. Hence we get $W=0$, i.e., $M$ is a compact conformally
flat manifold with constant positive scalar
curvature. By (2.5) and (3.7), we have
\begin{equation}\int_M [R-C_1(n)(n-1)\sqrt{\frac{4}{n-2}}|\mathring{Ric}|]|\mathring{Ric}|=0.\end{equation}
Since $M$ is a compact conformally
flat manifold with constant positive scalar
curvature, based on (2.11), proceeding as in the proof of (3.7), we obtain
\begin{equation}\int_M [R-\sqrt{n(n-1)}|\mathring{Ric}|]|\mathring{Ric}|\leq0.\end{equation}
By comparing (3.8) with (3.9), we obtain $\mathring{Ric}=0$, i.e., $M$ is a Einstein manifold. Hence $M$ is isometric to a quotient of the round $\mathbb{S}^n$.
\end{proof}

\begin{proof}[{\bf Proof of corollary
1.3}] When $\mathring{Rm}<\frac{R}{(n-1)C_1(n)}$, by (2.10) and the maximum principle, we get $\mathring{Rm}=0$. Thus $M$ is isometric to a quotient of the round $\mathbb{S}^n$.

When $\mathring{Rm}=\frac{R}{(n-1)C_1(n)}$, by Theorem 1.2, we have that $M$ is isometric to a quotient of the round $\mathbb{S}^n$.
\end{proof}

\begin{proof}[{\bf Proof of Theorem
1.4}]
By (2.20) and Lemma 2.5, we get
\begin{equation}
\begin{split}
\frac 12\triangle|\mathring{Ric}|^2&\geq\frac{n+2}{n}|\nabla |\mathring{Ric}||^2-\sqrt{\frac {n-2}{2(n-1)}}\left(|W|^2+\frac{2n}{n-2}|\mathring{Ric}|^2\right)^{\frac 12}|\mathring{Ric}|^2
+\frac{R}{n-1}|\mathring{Ric}|^2.
\end{split}
\end{equation}
We rewrite (3.10) as
\begin{equation}
\begin{split}
|\mathring{Ric}|\triangle|\mathring{Ric}|&\geq\frac{2}{n}|\nabla |\mathring{Ric}||^2-\sqrt{\frac {n-2}{2(n-1)}}\left(|W|^2+\frac{2n}{n-2}|\mathring{Ric}|^2\right)^{\frac 12}|\mathring{Ric}|^2
+\frac{R}{n-1}|\mathring{Ric}|^2
\end{split}
\end{equation}
in the sense of distributions.

Set $u=|\mathring{Ric}|$.
By (3.11), we compute
\begin{equation}
\begin{split}
u^{\gamma}\triangle u^{\gamma}&=u^{\gamma}\left(\gamma(\gamma-1)u^{\gamma-2}|\nabla u|^2+\gamma u^{\gamma-1}\triangle u\right)\\
&=\frac{\gamma-1}{\gamma}|\nabla u^{\gamma}|^2+\gamma
u^{2\gamma-2}u\triangle u\\
&\geq(1-\frac{n-2}{n\gamma})|\nabla u^{\gamma}|^2-\sqrt{\frac {n-2}{2(n-1)}}\gamma\left(|W|^2+\frac{2n}{n-2}u^2\right)^{\frac 12}u^{2\gamma}+\frac{R}{n-1}\gamma u^{2\gamma}.
\end{split}
\end{equation}
Integrating (3.12) by parts over $M^n$, it follows that
\begin{equation}
\begin{split}
0\geq(2-\frac{n-2}{n\gamma})\int_M |\nabla u^{\gamma}|^2+\gamma\frac{R}{n-1}\int_M u^{2\gamma}
-\sqrt{\frac {n-2}{2(n-1)}}\gamma\int_M \left(|W|^2+\frac{2n}{n-2}u^2\right)^{\frac 12}u^{2\gamma}.
\end{split}
\end{equation}
For $2-\frac{n-2}{n\gamma}>0$, by the definition of Yamabe constant and (3.13), we get
\begin{equation}
\begin{split}
0\geq(2-\frac{n-2}{n\gamma})\frac{n-2}{4(n-1)}\mu([g])\left(\int_M  u^{\frac{2n\gamma}{n-2}}\right)^{\frac{n-2}{n}}-\sqrt{\frac {n-2}{2(n-1)}}\gamma\int_M \left(|W|^2+\frac{2n}{n-2}u^2\right)^{\frac 12}u^{2\gamma}\\
+\frac {4n\gamma+\frac{1}{\gamma}(n-2)^2-2n(n-2)}{4n(n-1)}R\int_M u^{2\gamma}.
\end{split}
\end{equation}
 By H\"{o}lder inequality,  we obtain
\begin{equation}
\begin{split}
0\geq\left[(2-\frac{n-2}{n\gamma})\frac{n-2}{4(n-1)}\mu([g])-\sqrt{\frac {n-2}{2(n-1)}}\gamma\left(\int_M \left(|W|^2+\frac{2n}{n-2}u^2\right)^{\frac n4}\right)^{\frac 2n}\right]
\left(\int_M  u^{\frac{2n\gamma}{n-2}}\right)^{\frac{n-2}{n}}\\
+\frac {4n\gamma+\frac{1}{\gamma}(n-2)^2-2n(n-2)}{4n(n-1)}R\int_M u^{2\gamma}.
\end{split}
\end{equation}

Taking $\gamma=\frac{(n-2)(1+\sqrt{1-\frac 4n})}{4}$. From (3.15), we get
\begin{equation}
0\geq\left[\frac{\sqrt{2}}{\sqrt{(n-2)(n-1)}}\mu([g])-\left(\int_M (|W|^2
+\frac{2n}{n-2}u^2)^{\frac n4}\right)^{\frac 2n}\right]
\left(\int_M  u^{\frac{2n\gamma}{n-2}}\right)^{\frac{n-2}{n}}.
\end{equation}
Since $W$ is totally trace-free, one has
$$|W+\frac{\sqrt{n}}{\sqrt{2}(n-2)}\mathring{Ric}\circledwedge g|^2=|W|^2+\frac{2n}{n-2}|\mathring{Ric}|^2$$
and the pinching condition (1.1) implies that $M$ is Einstein.

In particular, we choose $\gamma$ such that $\frac {4n\gamma+\frac{1}{\gamma}(n-2)^2-2n(n-2)}{4n(n-1)}>0$, and
\begin{equation*}
0\geq\left[\frac{2}{C_2(n)n}\mu([g])-\left(\int_M \left(|W|^2+\frac{2n}{n-2}u^2\right)^{\frac n4}\right)^{\frac 2n}\right]
\left(\int_M  u^{\frac{2n\gamma}{n-2}}\right)^{\frac{n-2}{n}}.
\end{equation*}
From the above, the pinching condition (1.1) implies that $M$ is Einstein.
Hence, the pinching condition (1.1) implies
\begin{equation}\left(\int_M |W|^{\frac n2}\right)^{\frac 2n}<\frac{2}{C_2(n)n}\mu([g]).\end{equation}
By the  rigidity result for positively curved Einstein manifolds (see Theorem 1.1 of \cite{FX3}),
(3.17) implies that $M$ is isometric to a quotient of the round $\mathbb{S}^n$.
\end{proof}

\begin{proof}[{\bf Proofs of corollary
1.6 and Remark 1.7}]To prove Corollary 1.6, we need the
following lower bound for the Yamabe invariant on  compact four-dimensional manifolds which was proved by M. J. Gursky (see \cite{{Ca2},{G}}):
\begin{equation*}\int_M R^2-12\int_M|\mathring{Ric}|^2\leq{\mu([g])}^2,\end{equation*}
the inequality is strict unless $(M^4, g)$ is conformally Einstein. From this inequality, we get
$$\int_M|W|^{2}+4\int_M|\mathring{Ric}|^2
-\frac{1}{3}{\mu([g])}^2\leq\int_M|W|^{2}+8\int_M|\mathring{Ric}|^2
-\frac{1}{3}\int_M R^2$$
and
$$\int_M|W|^{2}+4\int_M|\mathring{Ric}|^2
-\frac{1}{6}{\mu([g])}^2\leq\int_M|W|^{2}+6\int_M|\mathring{Ric}|^2
-\frac{1}{6}\int_M R^2$$
Moreover, the two inequalities are strict unless $(M^4, g)$ is conformally Einstein. In the first case ``$<$'',
Theorem 1.4 immediately implies Corollary 1.6. In the second case ``$=$'',  $g$ is conformally
Einstein. Since $g$ has constant scalar curvature, $g$ is Einstein from the proof of Obata Theorem (see \cite{LP}). Hence $\frac{1}{6}\int_M R^2=\frac{1}{6}{\mu([g])}^2$, by Corollary 1.16 in \cite{F}, we complete the proof of this corollary.

By the Chern-Gauss-Bonnet formula ( see Equation 6.31 of \cite{B})
$$\int_M|W|^{2}-2\int_M|\mathring{Ric}|^2+\frac{1}{6}\int_M R^2
=32{\pi}^2\chi(M),$$
the right-hand sides can be
written as
$$\int_M|W|^{2}+8\int_M|\mathring{Ric}|^2
-\frac{1}{3}\int_M R^2=5\int_M|W|^{2}+\frac {1}{3}\int_M R^2
-128{\pi}^2\chi(M)$$
and$$\int_M|W|^{2}+6\int_M|\mathring{Ric}|^2
-\frac{1}{6}\int_M R^2=2\int_M|\mathring{Rm}|^{2}
-32{\pi}^2\chi(M).$$
This proves Remark 1.7.
\end{proof}

\begin{proof}[{\bf Proof of Theorem
1.8}]By (1.4), all inequalities leading to (3.16)
become equalities. Hence $c(|W|^2+\frac{2n}{n-2}|\mathring{Ric}|^2)^{\frac n4}=|\mathring{Ric}|^{\frac{2n\gamma}{n-2}}$, where $c$ is constant, and at every point either $\mathring{Ric}$ is null, i.e., $M$ is Eninstein, or it has an eigenvalue of
multiplicity $(n-1)$ and another of multiplicity $1$ for $|\nabla \mathring{Ric}|^2=\frac{n+2}{n}|\nabla |\mathring{Ric}||^2$. Since $M$ has harmonic curvature,
$M$ must be real analytic
in suitable (harmonic) local coordinates. So we have $W_{kijl}\mathring{R}_{ij}\mathring{R}_{kl}=0$. From (2.20), we get
\begin{equation*}
\triangle|\mathring{Ric}|^2=2|\nabla \mathring{Ric}|^2+2\frac{n}{n-2}\mathring{R}_{ij}\mathring{R}_{jl}\mathring{R}_{li}
+2\frac{R}{n-1}|\mathring{Ric}|^2.
\end{equation*}
Based on the above equality, using the same argument as in the proof of (3.15), we can obtain that if $$\left(\int_M |\mathring{Ric}|^{\frac n2}\right)^{\frac 2n}<\frac {1}{\sqrt{n(n-1)}}\mu([g]),$$ then $M$ is Einstein.

Case 1 When $W\neq0$, (1.4) implies that $\left(\int_M |\mathring{Ric}|^{\frac n2}\right)^{\frac 2n}<\frac {1}{\sqrt{n(n-1)}}\mu([g])$. By the above result, $M$ is Einstein.

Case 2 When $W\equiv0$, $M$ is locally conformally flat. Suppose that the Ricci tensor has an eigenvalue of multiplicity $(n-1)$ and another of multiplicity $1$.
If the Ricci tensor is parallel, by the de Rham decomposition Theorem,
$M$ is covered isometrically by the product of Einstein manifolds. We have $R=\sqrt{n(n-1)}|\mathring{Ric}|$. Since $M$ has positive scalar curvature, then the only possibility is that $M$ is covered
isometrically by $\mathbb{S}^1\times \mathbb{S}^{n-1}$ with the product metric.
On the other hand, if the Ricci tensor is not parallel, by the  classification
result of Derdzi\'{n}ski , this concludes the proof of Theorem 1.8.
\end{proof}

\begin{proof}[{\bf Proof of Theorem
1.9}]Based on Lemma 2.4, using the same argument as in the proof of (3.15), we can obtain
\begin{equation}
\begin{split}
0\geq\left[(2-\frac{1}{\gamma})\frac{n-2}{4(n-1)}\mu([g])-C_2(n)\gamma\left(\int_M |W|^\frac{n}{2}\right)^\frac{2}{n}-2\sqrt{\frac{n-1}{n}}\gamma\left(\int_M|\mathring{Ric}|^\frac{n}{2}\right)^{\frac 2n}\right]
\left(\int_M  |W|^{\frac{2n\gamma}{n-2}}\right)^{\frac{n-2}{n}}\\
+\frac {8(n-1)\gamma+\frac{1}{\gamma}n(n-2)-2n(n-2)}{4n(n-1)}R\int_M |W|^{2\gamma}.
\end{split}
\end{equation}

Case 1. $4\leq n\leq9$, choose $\gamma=1$. It follows from (3.18) that
\begin{equation}
\begin{split}
0\geq\left[\frac{n-2}{4(n-1)}\mu([g])-C_2(n)\left(\int_M |W|^\frac{n}{2}\right)^\frac{2}{n}-2\sqrt{\frac{n-1}{n}}\left(\int_M|\mathring{Ric}|^{\frac{n}{2}}\right)^{\frac 2n}\right]
\left(\int_M  |W|^{\frac{2n}{n-2}}\right)^{\frac{n-2}{n}}\\
+\frac {10n-8-n^2}{4n(n-1)}R\int_M |W|^{2}.
\end{split}
\end{equation}
From (3.19), the pinching condition (1.5) implies that $M$ is locally conformal flat.Moreover,  we get the same conclusion if we assume just the weak inequality in (1.5).
Hence, the pinching condition (1.5) implies
\begin{equation}\left(\int_M |\mathring{Ric}|^{\frac n2}\right)^{\frac 2n}<\frac {(n-2)\sqrt{n}}{8(n-1)^{\frac 23}}\mu([g])<\frac {1}{\sqrt{n(n-1)}}\mu([g]).\end{equation}

Case 2. $n\geq10$, choose $\frac 1\gamma=1+\sqrt{1-\frac {8(n-1)}{(n-2)n}}$. From (3.18), we get
\begin{equation}
0\geq\left[\frac{2}{n}\mu([g])-C_2(n)\left(\int_M |W|^\frac{n}{2}\right)^\frac{2}{n}-2\sqrt{\frac{n-1}{n}}\left(\int_M|\mathring{Ric}|^{\frac{n}{2}}\right)^{\frac 2n}\right]
\left(\int_M  |W|^{\frac{2n\gamma}{n-2}}\right)^{\frac{n-2}{n}}.
\end{equation}
From (3.21), the pinching condition (1.5) implies that $M$ is locally conformal flat. Hence, the pinching condition (1.5) implies
\begin{equation}\left(\int_M |\mathring{Ric}|^{\frac n2}\right)^{\frac 2n}<\frac {1}{\sqrt{n(n-1)}}\mu([g]).\end{equation}

By the  rigidity result for locally conformal flat manifolds (see Theorem 1.13 of \cite{FX}), noting that the difference between the Yamabe constants in \cite{FX} and this paper is $\frac{4(n-1)}{n-2}$,
(3.20) and (3.22) imply that $M$ is isometric to a quotient of the round $\mathbb{S}^n$.
\end{proof}

\begin{proof}[{\bf Proof of Theorem
1.11}]By (1.6), all inequalities leading to (3.18)
become equalities. Hence at every point, $W=0$ or   either $\mathring{Ric}$ is null, i.e., $M$ is Eninstein, or it has an eigenvalue of
multiplicity $(n-1)$ and another of multiplicity $1$ for $W_{ijkl}W_{ijkh}\mathring{Ric}_{hl}=\sqrt{\frac{n-1}{n}}|\mathring{Ric}||W|^2$. By the same argument as in the proof of Theorem 1.8, we can obtain that if $$\left(\int_M |\mathring{Ric}|^{\frac n2}\right)^{\frac 2n}<\frac {1}{\sqrt{n(n-1)}}\mu([g]),$$ then $M$ is Einstein.

Case 1 When $W\neq0$, (1.6) implies that $\left(\int_M |\mathring{Ric}|^{\frac n2}\right)^{\frac 2n}<\frac {1}{\sqrt{n(n-1)}}\mu([g])$. By the above result, $M$ is Einstein.

Case 2 When $W\equiv0$, $M$ is locally conformally flat. (1.6) implies that (1.4) holds.
By Theorem 1.8 , this concludes the proof of Theorem 1.11.
\end{proof}

\bibliographystyle{amsplain}

\end{document}